\documentclass[review]{elsarticle}
\usepackage{amsfonts,epsf,amsmath,amssymb,graphicx}
\usepackage{epstopdf}
\usepackage[top=2.5cm,bottom=2.5cm,left=2.5cm,right=2.5cm]{geometry}
\usepackage{lineno,hyperref}
\modulolinenumbers[5]

\newtheorem{theorem}{\bf Theorem}[section]
\newtheorem{corollary}[theorem]{\bf Corollary}
\newtheorem{lemma}[theorem]{\bf Lemma}
\newtheorem{proposition}[theorem]{\bf Proposition}

\newtheorem{problem}[theorem]{\bf Problem}

\journal{Applied Mathematics and Computation}

\begin{document}

\begin{frontmatter}

\title{The Szeged Index and the Wiener Index of Partial Cubes with Applications to Chemical Graphs}

%\author{Niko Tratnik}
%\author{Petra \v Zigert Pleter\v sek}


\author[mymainaddress,mysecondaryaddress,marusic]{Matev\v z \v Crepnjak}
\ead{matevz.crepnjak@um.si}
\author[mymainaddress]{Niko Tratnik}
\ead{niko.tratnik@um.si}

\address[mymainaddress]{Faculty of Natural Sciences and Mathematics, University of Maribor, Slovenia}
\address[mysecondaryaddress]{Faculty of Chemistry and Chemical Engineering, University of Maribor, Slovenia}
\address[marusic]{Andrej Maru\v si\v c Institute, University of Primorska, Slovenia}

\begin{abstract}
In this paper we study the Szeged index of partial cubes and hence generalize the result proved by V. Chepoi and S. Klav\v zar, who calculated this index for benzenoid systems.
It is proved that the problem of calculating the Szeged index of a partial cube can be reduced to the problem of calculating the Szeged indices of weighted quotient graphs with respect to a partition coarser than $\Theta$-partition. Similar result for the Wiener index was recently proved by S. Klav\v zar and M. J. Nadjafi-Arani. Furthermore, we show that such quotient graphs of partial cubes are again partial cubes. Since the results can be used to efficiently calculate the Wiener index and the Szeged index for specific families of chemical graphs, we consider $C_4C_8$ systems and show that the two indices of these graphs can be computed in linear time. 
\end{abstract}


\end{frontmatter}

\linenumbers

\section{Introduction}
In the present paper the Szeged index and the Wiener index of partial cubes are investigated. Our main result extends a parallel result about the Szeged index of benzenoid systems \cite{chepoi-1997} to all partial cubes. In \cite{chepoi-1997} the corresponding indices were expressed as the sum of indices of three weighted quotient trees, while in the general case these trees are quotient graphs of a partial cube with respect to a partition coarser than $\Theta$-partition.

Partial cubes constitute a large class
of graphs with a lot of applications and includes, for example, many families of chemical graphs (benzenoid systems, trees, $C_4C_8$ systems, phenylenes, cyclic phenylenes, polyphenylenes). Therefore, our results can be used to calculate the indices of a particular family of chemical graphs in linear time.

The Wiener index and the Szeged index are some of the most commonly studied topological indices. Their history goes back to $1947$, when H. Wiener used the distances in the molecular graphs of alkanes to calculate their boiling points \cite{wiener}. This research has led to the Wiener index, which is defined as
$$W(G) = \sum_{\lbrace u, v \rbrace \subseteq V(G)} d_G(u,v).$$

\noindent
In that paper it was also noticed that for every tree $T$ it holds
$$W(T) = \sum_{e =uv \in E(T)}n_u(e)n_v(e),$$
where $n_u(e)$ denotes the number of vertices of $T$ whose distance to $u$ is smaller than the distance to $v$ and $n_v(e)$ denotes the number of vertices of $T$ whose distance to $v$ is smaller than the distance to $u$. Therefore, a proper generalization to any graph was defined in \cite{gut_sz} as

$$Sz(G) = \sum_{e=uv \in E(G)}n_u(e)n_v(e).$$
\noindent
This topological index was later named as the Szeged index. For some recent results on the Wiener index and the Szeged index see \cite{fozan,GKS,knor2,knor,LY,RM}. Moreover, some other variants of the Wiener index were introduced, for example Wiener polarity index (see \cite{DLS,LLSW,MSY}).

\noindent
The Wiener index, due to its correlation with a large number of physico-chemical properties of organic molecules and its interesting mathematical properties, has been extensively studied in both theoretical and chemical literature. Later, the Szeged index was introduced and it was shown that it also has many applications, for example in drug modelling \cite{drug} and in networks \cite{klavzar-2013,pisanski}.

The paper reads as follows. In the next section we give some basic definitions needed later. In Section \ref{main} we prove that the problem of calculating the Szeged index of a partial cube can be reduced to the problem of calculating the Szeged indices of weighted quotient graphs with respect to a partition coarser than $\Theta$-partition. It turns out that such quotient graphs are again partial cubes. We also consider a similar result for the Wiener index, which was proved in \cite{klavzar-2016} under the name partition distance. In Section \ref{uporaba} we show that the mentioned results generalize already known results for benzenoid systems (see \cite{chepoi-1996,chepoi-1997}) and apply them to $C_4C_8$ systems. Furthermore, we demonstrate with an example that the method can also be used to calculate the corresponding indices by hand.

%%%%%%%%%%%%%%%%%%%%%%%%%%%%%%%%%%%%%%%%%%%%%%%%%%%%%%%%%%%%%%%%%%%%%
%%%%%%%%%%%%%%%%%%%%%%%%%%%%%%%%%%%%%%%%%%%%%%%%%%%%%%%%%%%%%%%%%%%%%
\section{Preliminaries}
%%%%%%%%%%%%%%%%%%%%%%%%%%%%%%%%%%%%%%%%%%%%%%%%%%%%%%%%%%%%%%%%%%%%%
%%%%%%%%%%%%%%%%%%%%%%%%%%%%%%%%%%%%%%%%%%%%%%%%%%%%%%%%%%%%%%%%%%%%%
The {\em distance} $d_G(x,y)$ between vertices $x$ and $y$ of a graph $G$ is the length of a shortest path between vertices $x$ and $y$ in $G$. We also write $d(x,y)$ for $d_G(x,y)$.
\bigskip

\noindent
Two edges $e_1 = u_1 v_1$ and $e_2 = u_2 v_2$ of graph $G$ are in relation $\Theta$, $e_1 \Theta e_2$, if
$$d_G(u_1,u_2) + d_G(v_1,v_2) \neq d_G(u_1,v_2) + d_G(u_1,v_2).$$
Note that this relation is also known as Djokovi\' c-Winkler relation (see \cite{djoko,winki}).
The relation $\Theta$ is reflexive and symmetric, but not necessarily transitive.
We denote its transitive closure (i.e.\ the smallest transitive relation containing $\Theta$) by $\Theta^*$. Let $ \mathcal{E} = \lbrace E_1, \ldots, E_r \rbrace$ be the $\Theta^*$-partition of the set $E(G)$. Then we say that a partition $\lbrace F_1, \ldots, F_k \rbrace$ of $E(G)$ is \textit{coarser} than $\mathcal{E}$
if each set $F_i$ is the union of one or more $\Theta^*$-classes of $G$ (see \cite{klavzar-2016}).
\bigskip

\noindent
The {\em hypercube} $Q_n$ of dimension $n$ is defined in the following way: 
all vertices of $Q_n$ are presented as $n$-tuples $(x_1,x_2,\ldots,x_n)$ where $x_i \in \{0,1\}$ for each $1\leq i\leq n$ 
and two vertices of $Q_n$ are adjacent if the corresponding $n$-tuples differ in precisely one coordinate. Therefore, the \textit{Hamming distance} between two tuples $x$ and $y$ is the number of positions in $x$ and $y$ in which they differ.
\bigskip

\noindent
Let $G$ be a graph and $e=uv$ an edge of $G$. Throughout the paper we will use the following notation:
$$N_1(e|G) = \lbrace x \in V(G) \ | \ d_G(x,u) < d_G(x,v) \rbrace, $$
$$N_2(e|G) = \lbrace x \in V(G) \ | \ d_G(x,v) < d_G(x,u) \rbrace. $$
Also, for $s \in \lbrace 1,2\rbrace$, let
$$n_s(e|G) = |N_s(e|G)|.$$

\noindent
A subgraph $H$ of $G$ is called an \textit{isometric subgraph} if for each $u,v \in V(H)$ it holds $d_H(u,v) = d_G(u,v)$. Any isometric subgraph of a hypercube is called a {\em partial cube}. We write $\langle S \rangle$ for the subgraph of $G$ induced by $S \subseteq V(G)$. The following theorem puts forth two fundamental characterizations of partial cubes, cf. \cite{klavzar-book}:
\begin{theorem}  \label{th:partial-k} For a connected graph $G$, the following statements are equivalent:
\begin{itemize}
\item [(i)] $G$ is a partial cube.
\item [(ii)] $G$ is bipartite, and $\langle N_1(e|G) \rangle $ and $\langle N_2(e|G) \rangle$ are convex subgraphs of $G$ for all $e \in E(G)$.
\item [(iii)] $G$ is bipartite and $\Theta = \Theta^*$.
\end{itemize}
\end{theorem}

\noindent Note that the characterization $(ii)$ is due to Djokovi\' c (1973), and $(iii)$ to Winkler (1984).
Is it also known that if $G$ is a partial cube and $E$ is a $\Theta$-class of $G$, then $G - E$ has exactly two connected components, namely $\langle N_1(e|G) \rangle $ and $\langle N_2(e|G) \rangle$, where $ab \in E$. For more information about partial cubes see \cite{klavzar-book}. 
\bigskip

\noindent
Let $H$ and $G$ be arbitrary graphs. Then a mapping $\alpha : V(H) \to V(G)$ is
an {\it isometric embedding}, if $d_H(u,v) = d_G(\alpha(u),\alpha(v))$ 
for each $u,v \in V(H)$.

\noindent
Suppose $\Pi$ is a partition of the vertex set of a graph $G$. The \textit{quotient graph} $G / \Pi$ is a graph with
vertex set $\Pi$, and for which distinct classes $C_1,C_2 \in \Pi $ are adjacent if some vertex in $C_1$ is
adjacent to a vertex of $C_2$. If $F \subseteq E(G)$, let $\Pi_F$ be the partition of $V(G)$ whose classes are the vertices of the connected components of $G - F$. The graph $G / \Pi_F$ is also denoted by $G / F$.
\bigskip

%\noindent
%The \textit{canonical embedding} of a connected graph $G$ is defined as follows: Let the $\Theta^{*}$ classes
%of $G$ be $E_1,E_2,\ldots,E_k$, and for each index $i$, put $G_i = G - E_i$. Let $\Pi_i$ be the partition of
%$V(G_i)$ whose classes are the vertices of the connected components of $G_i$. Let $\alpha_i: G \rightarrow G/\Pi_i$
%be the map sending any $v$ to the component of $G_i$ that contains it. The canonical embedding
%$$\alpha': V(G) \rightarrow V(G/ \Pi_1 \ \Box \ G/  \Pi_2 \ \Box \cdots \Box \ G/ \Pi_k)$$
%is defined by
%$$\alpha'(v) = \big( \alpha_1(v), \alpha_2(v), \ldots, \alpha_k(v)\big).$$
%Graham and Winker proved that the canonical embedding is an isometric embedding. For the details see \cite{klavzar-book}.
%\bigskip

%\noindent
%The {\em Wiener index} of a graph $G$ is defined as $\displaystyle{W(G) = \frac{1}{2} \sum_{u \in V(G)} \sum_{v \in V(G)} d_G(u,v)}$.

\noindent Using the notation defined above, the {\em Szeged index} of a graph $G$ is defined as $$Sz(G) =\sum_{e \in E(G)}n_1(e|G) n_2(e|G).$$

\noindent
Now we extend the definitions of the Wiener index and the Szeged index to weighted graphs as follows. Let $G$ be a graph and let $w:V(G)\rightarrow {\mathbb R}^+$  and $w':E(G)\rightarrow {\mathbb R}^+$ be given functions. Then $(G,w)$ and $(G,w,w')$ are a {\em vertex-weighted graph} and a {\em vertex-edge-weighted graph}, respectively. The \textit{vertex-weighted Wiener index} \cite{klavzar-1997} of $(G,w)$ is defined as
$$W(G,w) = \frac{1}{2} \sum_{u \in V(G)} \sum_{v \in V(G)} w(u)w(v)d_G(u,v).$$

\noindent
Furthermore, if $e=uv$ is an edge of $G$ and $s \in \lbrace 1, 2\rbrace$, set
\begin{eqnarray*}
n_s(e|(G,w)) & = & \sum_{x \in N_s(e|G)} w(x).
\end{eqnarray*}

\noindent
The \textit{Szeged index of a vertex-edge-weighted graph} \cite{tratnik} is defined as
\begin{eqnarray*}
Sz(G,w,w') & = & \sum_{e \in E(G)} w'(e)n_1(e|(G,w))n_2(e|(G,w)). 
\end{eqnarray*}
If $G$ is a partial cube and $E$ a $\Theta$-class, then for any two edges $e,f \in E$ and $s \in \lbrace 1, 2\rbrace$ it holds
$$N_s(e|G) = N_s(f|G).$$
Therefore, it is legitimate to define
$$n_s(E|G) = n_s(e|G),$$
$$n_s(E|(G,w)) = n_s(e|(G,w)),$$
where $e \in E$.

\section{The Szeged index and the Wiener index of partial cubes}
\label{main}

In this section we prove the main result of the paper. We start with the following basic proposition, the details can be found in \cite{klavzar-2015}.

\begin{proposition}\label{prop1}
Let $G$ be a partial cube with $n$ vertices and let $E_1, E_2, \ldots, E_r$ be its $\Theta$-classes. Then
$$W(G) = \sum_{i = 1}^{r}n_1(E_i|G)n_2(E_i|G)$$
and
$$Sz(G) = \sum_{i = 1}^{r}|E_i|n_1(E_i|G)n_2(E_i|G).$$
\end{proposition}

%\begin{proof}
%The proof for the Wiener index is based on the canonical embedding, for the details see the proof of \cite[Proposition 19.2]{klavzar-book}.
%
%\noindent
%Since for every two edges $e,f$ of the same $\Theta$-class of $G$ if holds (without loss of generality) $N_1(e|G) = N_1(f|G)$ and $N_2(e|G) = N_2(f|G)$, the proof for the Szeged index follows obviously. \qed
%\end{proof}

\noindent
The above proposition can be useful for deriving closed formulas for some families of graphs. But since the number of $\Theta$-classes of a partial cube can be very large, in the rest of this section we generalize this result such that the computation is possible for any partition of $E(G)$ coarser than $\Theta$-partition. As a consequence, the calculation of these two indices will be much more efficient. To prove the main theorem, some preparation is needed. The following lemma claims that every quotient graph obtained from a partition coarser than the $\Theta$-partition is again a partial cube. 
\begin{lemma}
\label{partial}
Let $G$ be a partial cube. If $\lbrace F_1, \ldots, F_k \rbrace$ is a partition coarser than the $\Theta$-partition, then $G / F_i$ is a partial cube for each $i$, $1 \leq i \leq k$.
\end{lemma}

\begin{proof}
Let $F_i$ be the union of $\Theta$-classes $E_{i,1}, \ldots, E_{i,r(i)}$ and $j \in \lbrace 1, \ldots, r(i) \rbrace$. Since $G$ is a partial cube, the graph $G - E_{i,j}$ has exactly two connected components; denote these two components by $C_0^{j}$ and $C_1^{j}$. We define function $\ell_j: V(G / F_i) \rightarrow \lbrace 0, 1\rbrace$ in the following way: 
$$\ell_j(x) = \begin{cases}
0; & x \text{ is a subgraph of } C_0^j \\
1; & x \text{ is a subgraph of } C_1^j
 \end{cases}$$
 for every $x \in V(G / F_i)$. \\
Let $\ell : V(G / F_i) \rightarrow \lbrace 0, 1\rbrace^{r(i)}$ be defined as
$$\ell(x) = (\ell_1(x), \ldots, \ell_{r(i)}(x))$$
for each $x \in V(G / F_i)$.\\

Let $x,y \in V(G / F_i)$. One can easily find a path from $x$ to $y$ in $G / F_i$ which has length $\sum_{j = 1}^{r(i)} |\ell_j(x) - \ell_j(y)|$. Therefore,

$$d_{G / F_i}(x,y) \leq \sum_{j = 1}^{r(i)} |\ell_j(x) - \ell_j(y)|.$$
To prove the other inequality, let $P$ be a shortest path from $x$ to $y$ in $G / F_i$. Suppose that $|E(P)| < \sum_{j = 1}^{r(i)} |\ell_j(x) - \ell_j(y)|$. It follows that there exists $j \in \lbrace 1, \ldots, r(i) \rbrace$ such that $\ell_j(x) \neq \ell_j(y)$ and no edge of $P$ corresponds to the $\Theta$-class $E_{i,j}$. We obtain that there exists a path $Q$ in $G$ from a vertex in $x$ to a vertex in $y$, such that no edge of $Q$ is contained in $E_{i,j}$. Therefore, $x$ and $y$ are in the same connected component of the graph $G - E_{i,j}$. This is a contradiction, since $\ell_j(x) \neq \ell_j(y)$. Hence, 
$$d_{G / F_i}(x,y) = \sum_{j = 1}^{r(i)} |\ell_j(x) - \ell_j(y)|,$$
which is the Hamming distance. Obviously, $\ell$ is an isometric embbeding into $r(i)$-dimensional hypercube. Therefore, $G / F_i$ is a partial cube. \qed
\end{proof}

Let $G$ be a partial cube and let $\lbrace F_1, \ldots, F_k \rbrace$ be a partition coarser than the $\Theta$-partition. We next extend the quotient graphs $G / F_i$, $i \in \lbrace 1, \ldots, k \rbrace$, to weighted graphs as follows: 
\begin{enumerate}
\item for $x \in V(G / F_i)$, let $w_i(x)$ be the number of vertices in the connected component $x$ of $G - F_i$,
\item for $xy \in E(G / F_i)$, let $w_i'(xy)$ be the number of edges in $F_i$ with one end-point in $x$ and another in $y$.
\end{enumerate}
Now we are ready to state the following lemma. 
\begin{lemma}
\label{lemica}
Let $G$ be a partial cube and $\lbrace F_1, \ldots, F_k \rbrace$ a partition coarser than the $\Theta$-partition. Moreover, let $(G / F_i, w_i, w_i')$ be the weighted quotient graph for some $i \in \{1, \ldots, k \}$ and $\ell$ the function defined in the proof of Lemma \ref{partial}. Let $F_i$ be the union of $\Theta$-classes $E_{i,1}, \ldots, E_{i,r(i)}$ and $j \in \lbrace 1, \ldots, r(i) \rbrace$. Then
$$U_{i,j} = \{ xy \in E(G / F_i) \, | \, \ell_j(x) \neq \ell_j(y) \text{ and } \ell_t(x) = \ell_t(y) \text{ where } t \neq j \}$$
is a $\Theta$-class of $G / F_i$ and
$$|E_{i,j}| = \sum_{f \in U_{i,j}} w'_i(f).$$
Moreover, if $s \in \{ 1,2 \}$, then
$$n_s(E_{i,j} | G) = \sum_{x \in N_s(U_{i,j}|(G / F_i))}w(x).$$

\end{lemma}

\begin{proof}
That the set $U_{i,j}$ is a $\Theta$-class of $G / F_i$ follows directly from the fact that $\ell$ is the Hamming distance and $G / F_i$ is a partial cube (see Lemma \ref{partial}).

Every edge $f \in U_{i,j}$ represents $w'(f)$ edges of $E_{i,j}$. Also, any edge $e \in E_{i,j}$ is represented by some edge $f \in U_{i,j}$, since each edge of $G$ is in exactly one $\Theta$-class of $G$. Hence,
$$|E_{i,j}| = \sum_{f \in U_{i,j}} w'_i(f).$$

Let $s \in \{ 1,2 \}$. Every vertex $x \in N_s(U_{i,j}|(G / F_i))$ represents $w(x)$ vertices of $N_s(E_{i,j}|G)$. Conversely, any vertex $v \in N_s(E_{i,j}|G)$ is represented by exactly one vertex $x \in N_s(U_{i,j}|(G / F_i))$. Therefore,
$$n_s(E_{i,j} | G) = \sum_{x \in N_s(U_{i,j}|(G / F_i))}w(x).$$
\qed
\end{proof}

\noindent
The main result of this paper now follows from Lemma \ref{lemica}.

\begin{theorem}
\label{segi}
Let $G$ be a partial cube. If $\lbrace F_1, \ldots, F_k \rbrace$ is a partition coarser than the $\Theta$-partition, then
$$Sz(G) = \sum_{i=1}^k Sz(G / F_i, w_i, w_i').$$
\end{theorem}

\begin{proof}
Let all the notation be the same as in Lemma \ref{lemica}. Recall that $\mathcal{E}$ is a $\Theta$-partition of the set $E(G)$. By Proposition \ref{prop1} we obtain
\begin{eqnarray*}
Sz(G) & = & \sum_{e \in E(G)}n_1(e|G)n_2(e|G) \\
& = & \sum_{E \in \mathcal{E}}|E| n_1(E|G)n_2(E|G) \\
& = & \sum_{i=1}^k \sum_{j=1}^{r(i)} |E_{i,j}|n_1(E_{i,j}|G)n_2(E_{i,j}|G).
\end{eqnarray*} 
 Using Lemma \ref{lemica} it follows
 \begin{eqnarray*}
Sz(G) & = & \sum_{i=1}^k \sum_{j=1}^{r(i)} \Bigg(\sum_{f \in U_{i,j}} w'_i(f) \Bigg) \cdot \Bigg( \sum_{x \in N_1(U_{i,j}|(G / F_i))}w(x) \Bigg) \cdot \Bigg( \sum_{x \in N_2(U_{i,j}|(G / F_i))}w(x) \Bigg) \\
& = & \sum_{i=1}^k \sum_{j=1}^{r(i)} \Bigg(\sum_{f \in U_{i,j}} w'_i(f) \Bigg) \cdot n_1(U_{i,j} | (G / F_i, w_i)) \cdot n_2(U_{i,j} | (G / F_i, w_i)) \\
& = & \sum_{i=1}^k \sum_{j=1}^{r(i)} \Bigg(\sum_{f \in U_{i,j}} w'_i(f)  \cdot n_1(U_{i,j} | (G / F_i, w_i)) \cdot n_2(U_{i,j} | (G / F_i, w_i)) \Bigg) \\
& = & \sum_{i=1}^k \Bigg( \sum_{j=1}^{r(i)} \sum_{f \in U_{i,j}} w'_i(f)  \cdot n_1(f | (G / F_i, w_i)) \cdot n_2(f | (G / F_i, w_i)) \Bigg).
\end{eqnarray*}
In the last step we have taken into account that $U_{i,j}$ is a $\Theta$-class and therefore, $$n_s(U_{i,j} | (G / F_i, w_i)) = n_s(f | (G / F_i, w_i))$$ for any $f \in U_{i,j}$ and $s \in \lbrace 1, 2\rbrace$. Since $\lbrace U_{i,1}, \ldots, U_{i,r(i)}\rbrace$ is a partition of the set $E(G / F_i)$, we finally deduce
\begin{eqnarray*}
Sz(G)& = & \sum_{i=1}^k \Bigg( \sum_{f \in E(G / F_i)} w'_i(f)  \cdot n_1(f | (G / F_i, w_i)) \cdot n_2(f | (G / F_i, w_i)) \Bigg) \\
& = & \sum_{i=1}^k Sz(G / F_i, w_i, w_i')
\end{eqnarray*} 
and the proof is complete. \qed
 
\end{proof}

At the end of this section we also consider the Wiener index of partial cubes. This problem was solved in \cite{klavzar-2016}, where it was shown that the Wiener index of an arbitrary connected vertex-weighted graph can be computed as the sum of the Wiener indices of weighted quotient graphs with respect to a partition coarser than $\Theta^*$-partition. Since for the partial cubes it holds $\Theta = \Theta^*$, we get the following result, which is a corollary of \cite[Theorem 3.3]{klavzar-2016}.

\begin{corollary} \cite{klavzar-2016}
\label{wini}
Let $G$ be a partial cube. If $\lbrace F_1, \ldots, F_k \rbrace$ is a partition coarser than the $\Theta$-partition, then
$$W(G) = \sum_{i=1}^k W(G / F_i, w_i).$$
\end{corollary}

%%%%%%%%%%%%%%%%%%%%%%%%%%%%%%%%%%%%%%%%%%%%%%%%%%%%%%%%%%%%%%%%%%%%%
%%%%%%%%%%%%%%%%%%%%%%%%%%%%%%%%%%%%%%%%%%%%%%%%%%%%%%%%%%%%%%%%%%%%%
\section{Applications to benzenoid systems and $C_4C_8$ systems}
\label{uporaba}
%%%%%%%%%%%%%%%%%%%%%%%%%%%%%%%%%%%%%%%%%%%%%%%%%%%%%%%%%%%%%%%%%%%%%
%%%%%%%%%%%%%%%%%%%%%%%%%%%%%%%%%%%%%%%%%%%%%%%%%%%%%%%%%%%%%%%%%%%%%

\noindent
It is known that the Wiener index and the Szeged index of a bezenoid system can be computed as the sum of the corresponding indices of weighted quotient trees, see \cite[Proposition 2]{chepoi-1997} and \cite[Proposition 3]{chepoi-1997}. Moreover, these trees are exactly quotient graphs with respect to a partition coarser than $\Theta$-partition, hence, these results also follows directly from the results from Section 3. Therefore, in this section we will concentrate to the $C_4C_8$ systems. More precisely, for $C_4C_8$ systems we show that the Wiener index and the Szeged index can be computed in linear time.

\noindent
Let ${\cal H}$ be the infinite $C_4C_8$ net and let $Z$ be a cycle in it (see Figure \ref{def}). Then a $C_4C_8$ {\em system} is induced by the vertices and edges of ${\cal H}$, lying on $Z$ and in its interior. Obviously, any $C_4C_8$ system is a bipartite graph, since the whole $C_4C_8$ net is bipartite.

\begin{figure}[h!] 
\begin{center}
\includegraphics[scale=0.8]{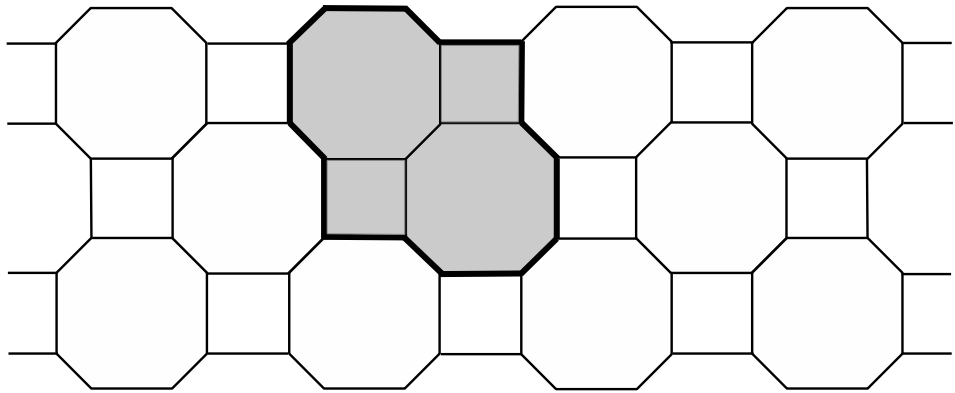}
\end{center}
\caption{\label{def} A cycle $Z$ and the corresponding $C_4C_8$ system.}
\end{figure}

$C_4C_8$ systems are important in chemical graph theory, since they are subgraphs of $TUC_4C_8(S)$ and $TUC_4C_8(R)$ nanotubes and nanotori, which are some of the central and most commonly studied families of chemical graphs (for example, see \cite{AY1,DS,HT1,AY2}). Also, the infinite version of a $C_4C_8$ system was considered in \cite{A}. In addition, $C_4C_8$ systems can represent chemical compounds, formed by the rings of cyclobutane and cyclooctane. 
%They are also molecular graphs, composed from cyclobutadiene and cyclooctatetraene (which are unsaturated cyclobutane and cyclooctane). For an example see Figure \ref{but-oct}. Some combinations of such rings were, for instance, studied in \cite{NOI}.

%\begin{figure}[h!] 
%\begin{center}
%\includegraphics[scale=0.8]{but-oct1.eps}
%\end{center}
%\caption{\label{but-oct} $C_4C_8$ system of cyclobuteno-cyclooctatetraene.}
%\end{figure}

%\noindent
%Let $c$ be a line segment that starts at
%the center of a peripheral edge of a $C_4C_8$ system $G$,
%goes orthogonal to it and ends at the first next peripheral
%edge of $G$. An {\it elementary cut} $C$ is the set of all the edges of $G$ which are intersected by $c$.
%\bigskip

As in the case of benzenoid systems, one can define elementary cuts as follows (the details on elementary cuts can be found elsewhere, for example, see \cite{gut}). An \textit{elementary cut} $C$ of a $C_4C_8$ system $G$ is a line segment that starts at
the center of a peripheral edge of a $C_4C_8$ system $G$,
goes orthogonal to it and ends at the first next peripheral
edge of $G$. By $C$ we sometimes also denote the set of edges that are intersected by the corresponding elementary cut. 
In the next theorem we show that each $C_4C_8$ system is a partial cube.

\begin{theorem}\label{partial_cube}
	Each $C_4C_8$ system is a partial cube.
\end{theorem}
\begin{proof}
Let $G$ be a $C_4C_8$ system. Obviously, $G$ is a bipartite graph. Let $e=ab \in E(G)$ and $C$ an elementary cut that contains $e$. It is easy to see that $G-C$ has exactly two connected components, namely $\langle N_1(e|G) \rangle$ and $\langle N_2(e|G)\rangle$.

To complete the proof, it suffices to show that the graph $\langle N_s(e|G) \rangle$, where $s \in \lbrace 1,2 \rbrace$, is convex. Let $x,y \in N_s(e|G)$ and let $P$ be a shortest path from $x$ to $y$. If $P$ contains an edge $f_1$ of $C$, then obviously contains also one more edge $f_2$ of $C$. Let $u_1 \in N_s(e|G)$ and $u_2 \in N_s(e|G)$ be the end-vertices of $f_1$ and $f_2$, respectively. Moreover, let $P_1$ be the part of $P$ from $x$ to $u_1$ and $P_2$ be a part of $P$ from $u_2$ to $y$. Then the path, composed of path $P_1$, the shortest path from $u_1$ to $u_2$, and path $P_2$ is a path from $x$ to $y$, that is shorter than $P$ - a contradiction. Therefore, it follows that $\langle N_s(e|G) \rangle$ is a convex subgraph of $G$. By Theorem \ref{th:partial-k}, the proof follows. \qed
\end{proof}

Note that from Theorem \ref{partial_cube} and Theorem \ref{th:partial-k} it follows that for each $C_4C_8$ system it holds $\Theta = \Theta^*$. The main insight for our consideration
is that every $\Theta$-class of a $C_4C_8$ system
$G$ coincides with exactly one of its elementary cuts.

\noindent
Let $G$ be a $C_4 C_8$ system and let $E_1, E_2,\ldots, E_r$ be its $\Theta$-classes. Without loss of generality, assume that $G$ contains at least one internal face of length $8$. Based on this assumption we have in $G$ four different directions of edges. Let $F_i$, $i \in \lbrace 1, \ldots, 4\rbrace$, be the union of all $\Theta$-classes with the same direction. Obviously, $\lbrace F_1, F_2, F_3, F_4 \rbrace$ is a partition coarser than the $\Theta$-partition.

\begin{proposition}
Let $G$ be a $C_4C_8$ system and let $\lbrace F_1, F_2, F_3, F_4 \rbrace$ be a partition of $E(G)$ defined as above. Then the graph $G/F_i$ is a tree for each $i \in \lbrace 1,2,3,4 \rbrace$.
\end{proposition}

\begin{proof}
Suppose that $G/F_i$ is not a tree and let $C$ be a cycle in $G/F_i$. Obviously, by Lemma \ref{partial} the graph $G/F_i$ is bipartite and therefore, the shortest possible cycle in it is of length $4$. 

Suppose that $C = X_1X_2X_3X_4$ is a cycle in $G/F_i$ of length $4$. Therefore, we obtain the situation in Figure \ref{cikel}. 

\begin{figure}[h!] 
\begin{center}
\includegraphics[scale=0.8]{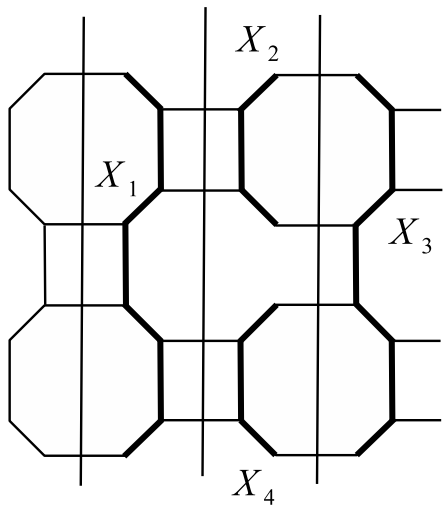}
\end{center}
\caption{\label{cikel} A cycle $C=X_1X_2X_3X_4$ in $G / F_i$.}
\end{figure}

Obviously, in such a case, the graph $G$ contains an internal face whose boundary is a cycle of length at least $10$, which is a contradiction. It is easy to see that the same holds if $C$ has length more than $4$. Therefore, there is no cycle  in $G/F_i$ and thus, $G/F_i$ is a tree. \qed
\end{proof}

Next, we extend the quotient trees $G/F_i$, $i \in \lbrace 1, 2, 3, 4 \rbrace$, to weighted trees $(G/F_i, w_i)$ and $(G/F_i,w_i,w_i')$ as in Section \ref{main}.

\noindent
The following result is a special case of Theorem \ref{segi} and Corollary \ref{wini}.

\begin{corollary}
\label{thm:Wiener}
If $G$ is a $C_4C_8$ system, then 
$$W(G)= W(G/F_1,w_1) + W(G/F_2,w_2) + W(G/F_3,w_3) + W(G/F_4,w_4)$$
and
$$Sz(G)= Sz(G/F_1,w_1,w_1') + Sz(G/F_2,w_2,w_2') + Sz(G/F_3,w_3,w_3') + Sz(G/F_4,w_4,w_4').$$
\end{corollary}

Our next goal is to prove that the Wiener index and the Szeged of a $C_4C_8$ system can be computed in $O(n)$ time. Since the weighted quotient trees can be computed in $O(n)$ time, it suffices to show that the Wiener indices and the Szeged indices of weighted trees can be computed in $O(n)$ time. In \cite{chepoi-1997} it was mentioned that if $(T,w)$ is a vertex-weighted tree, then the Wiener index can be computed as
\begin{equation} \label{wie_tree} W(T,w) = \sum_{e \in E(T)}n_1(e|T)n_2(e|T). 
\end{equation}
It was also shown that using this formula, the Wiener index of a weighted tree can be computed in linear time.

\noindent
Now we prove that the weighted Szeged index of a weighted tree can be computed in $O(n)$ time.

\begin{proposition}
\label{linear}
Let $(T,w,w')$ be a weighted tree with $n$ vertices. Then the weighted Szeged index $Sz(T,w,w')$ can be computed in linear time.
\end{proposition}

\begin{proof}
The proof is similar to the proofs from \cite{chepoi-1997}, \cite{kelenc} and \cite{tratnik}. For the sake of completeness, we give the proof anyway. 

Let $T$ be a rooted tree with a root $x$ and label the vertices of $T$ such that, if a vertex $y$ is labelled $\ell$, then all vertices in the subtree rooted at $y$ have labels smaller than $\ell$. Using the standard BFS (breadth-first search) algorithm this can be done in linear time. Next we visit all the vertices of $T$ according to this labelling (we start with the smallest label). Furthermore, for every vertex $y$, we will adopt weights $w(y)$ and calculate new weights $s(y)$.

\noindent
Assume that we are visiting some vertex $y \in V(T)$. The new weight $w(y)$ will be computed as the sum of all the weights of vertices in the subtree rooted at $y$. 
\begin{itemize}
\item If $y$ is a leaf, then $w(y)$ is left unchanged.
\item If $y$ is not a leaf, then update $w(y)$ by adding to it $w(z)$ for all down-neighbours $z$ of $y$.
\end{itemize}

\noindent
Obviously, for every vertex $y$ of the tree $T$, we can consider the subtree rooted at $y$ as a connected component of the graph $T - e$, where $e$ is the up-edge of $y$. Therefore, $n_1(e|T)=w(y)$. Let $n(T)=\sum_{u \in V(T)} w(u)$ (this can be computed in linear time). It follows that $n_2(e|T) = n(T) - w(y)$. Let $X$ be the sum of numbers $s(z)$ for all down-neighbours $z$ of $y$ (and $X=0$ if $y$ is a leaf). Finally, set $s(y) = X + w'(e)n_1(e|T) n_2(e|T)$ if $y \neq x$ and $s(y) = X$, if $y=x$. It is obvious that $s(x) = Sz(T,w,w')$ and the proof is complete. \qed
\end{proof}

\noindent
Finally, since the weighted quotient trees can be computed in $O(n)$ time, Corollary \ref{thm:Wiener}, Equation \ref{wie_tree}, and Proposition \ref{linear} imply the following theorem.

\begin{theorem}
The Wiener index and the Szeged index of a $C_4C_8$ system with $n$ vertices can be computed in $O(n)$ time.
\end{theorem}

It is already known that the Wiener index and the Szeged index of benzenoid systems can be computed in linear time \cite{klavzar-1997}. In the case of benzenoid systems and $C_4C_8$ systems the computation  in linear time is possible since the number of $\Theta$-classes, which have the same direction, is fixed and the corresponding quotient graphs are trees. In the general case, the quotient graphs may not be trees. Therefore, this technique can not be applied for partial cubes in general. 

Moreover, it is known that the Wiener index of benzenoid system can be computed in sub-linear time (for the details see \cite{chepoi-1998}). Therefore, we state the following open problem.

\begin{problem}
Let $G$ be a $C_4C_8$ system. Can the Wiener index and the Szeged index of $G$ be computed in sub-linear time.
\end{problem}

To conclude this section we demonstrate the procedure for computing the Wiener index and the Szeged index in the following example. Let $G$ be a $C_4 C_8$ system as in Figure \ref{ex:primer}.

\begin{figure}[h!] 
\begin{center}
\includegraphics[scale=0.6]{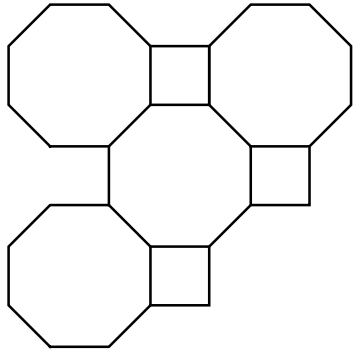}
\end{center}
\caption{\label{ex:primer} A $C_4C_8$ system.}
\end{figure}

\noindent
As we mentioned above, we get the following division of $G$ with respect to $\lbrace F_1,F_2,F_3,F_4 \rbrace$ 
(recall that $\lbrace F_1,F_2,F_3,F_4 \rbrace$ is a partition coarser than the $\Theta$-partition).

\noindent
Consequently, we obtain the weighted trees as shown in Figure \ref{ex:primer1} and Figure \ref{ex:primer2}.

\begin{figure}[h!] 
\begin{center}
\includegraphics[scale=0.7]{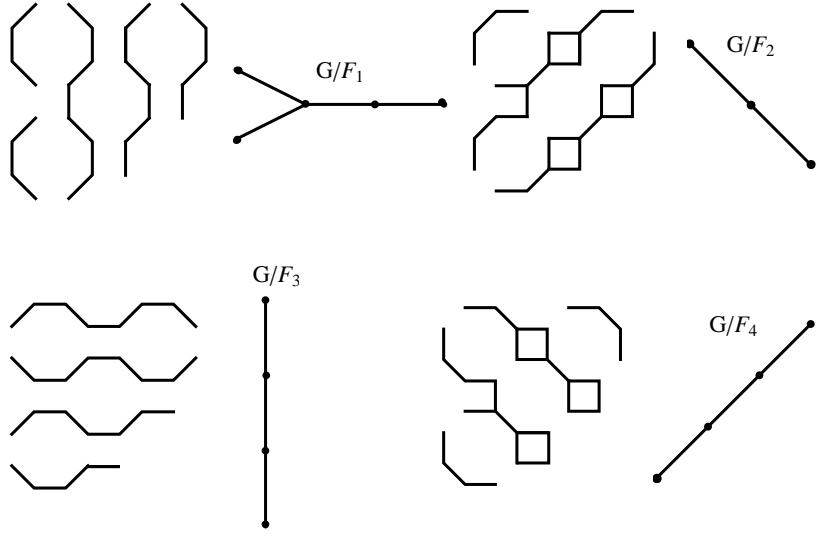}
\end{center}
\caption{\label{ex:primer1} A $C_4C_8$ system and the corresponding trees of a $C_4C_8$ system from Figure \ref{ex:primer}.}
\end{figure}

\begin{figure}[h!] 
\begin{center}
\includegraphics[scale=0.6]{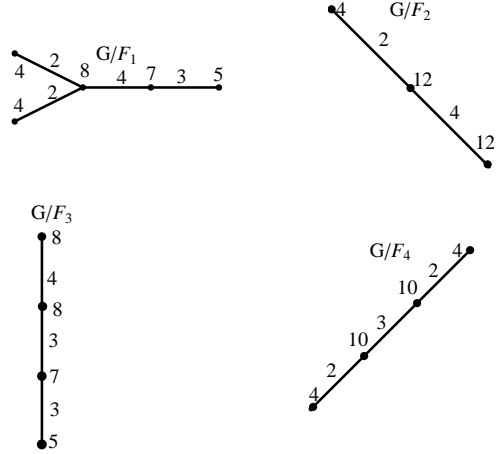}
\end{center}
\caption{\label{ex:primer2} The corresponding weighted trees $(G/F_1,w_1,w_1')$, $(G/F_2,w_2,w_2')$, $(G/F_3,w_3,w_3')$, and $(G/F_4,w_4,w_4')$ of a $C_4C_8$ system from Figure \ref{ex:primer}.}
\end{figure}

\noindent
Using Equation (\ref{wie_tree}) we obtain
\begin{align*}
	W(G/F_1,w_1) & = 4\cdot(4+8+7+5) + 4\cdot(4+8+7+5) + (4+4+8)\cdot(7+5) +  \\ 
						 & + (4+4+8+7)\cdot 5 = 499 \\
	W(G/F_2,w_2) & = 4\cdot(12+12) + (4+12) \cdot 12 = 288\\	
	W(G/F_3,w_3) & = 8\cdot(8+7+5) + (8+8)\cdot(7+5) + (8+8+7)\cdot 5 = 467 \\	
	W(G/F_4,w_4) & = 4\cdot(10+10+4) + (4+10)\cdot(10+4) + (4+10+10)\cdot 4 = 388
\end{align*}
and by Corollary \ref{thm:Wiener} it follows that
	$$W(G) = \sum_{i=1}^{4} W(G/F_i,w_i) = 1642.$$
Next,
\begin{align*}
	Sz(G/F_1,w_1,w_1') & = 2\cdot 4\cdot(4+8+7+5)+2\cdot 4\cdot(4+8+7+5)+4\cdot(4+8+4)\cdot(7+5) \\
	& + 3\cdot(4+4+8+7)\cdot5 = 1497\\
	Sz(G/F_2,w_2,w_2') & = 2\cdot 4 \cdot (12+12) + 4\cdot (4+12)\cdot 12 = 960\\	
	Sz(G/F_3,w_3,w_3') & = 4\cdot 8\cdot(8+7+5)+3\cdot(8+8)\cdot(7+5)+3\cdot(8+8+7)\cdot 5=1561 \\	
	Sz(G/F_4,w_4,w_4') & = 2\cdot 4\cdot(10+10+4)+3\cdot(4+10)\cdot(10+4)+2\cdot(4+10+10)\cdot 4 = 972
\end{align*}
and by Corollary \ref{thm:Wiener},
$$Sz(G) = \sum_{n=1}^{4} Sz(G/F_i,w_i,w_i') = 4990.$$

%%%%%%%%%%%%%%%%%%%%%%%%%%%%%%%%%%%%%%%%%%%%%%%%%%%%%%%%%%%%%%%%%%%%%
%%%%%%%%%%%%%%%%%%%%%%%%%%%%%%%%%%%%%%%%%%%%%%%%%%%%%%%%%%%%%%%%%%%%%

\section*{Acknowledgement} The author Matev\v z \v Crepnjak acknowledge the financial support from the Slovenian Research Agency (research core funding No. P1-0285). 

The author Niko Tratnik was finacially supported by the Slovenian Research Agency.

\newpage

\end{document}